\documentclass[a4paper,11pt]{article}
\usepackage[margin=1in]{geometry}
\usepackage{amsmath,amssymb,amsfonts,stmaryrd}
\usepackage{comment}
\usepackage{commath}
\usepackage{soul}
\usepackage{array}
\usepackage{empheq}
\usepackage{xfrac}
\usepackage{minibox}
\usepackage{enumitem}
	\setlist{nosep} %
\usepackage{color}
\usepackage{tikz} 
\usepackage[lofdepth,lotdepth]{subfig}

\usepackage[nocompress]{cite}
\usepackage[hidelinks]{hyperref}
\definecolor{darkgreen}{rgb}{0.0, 0.5, 0}
\hypersetup{
 colorlinks   = true, %
 urlcolor     = blue, %
 linkcolor    = darkgreen, %
 citecolor   = darkgreen %
}
\usepackage[nameinlink]{cleveref}

\newcommand{\CK}{{\cal N}}

\newcommand{\CR}{{\cal R}}
\newcommand{\CS}{{\cal S}}

\newcommand{\CV}{{\cal V}}

\newcommand{\la}{\langle}
\newcommand{\ra}{\rangle}

\usepackage{amsthm}
\newtheorem{theorem}{Theorem}

\newtheorem{corollary}[theorem]{Corollary}
\newtheorem{lemma}[theorem]{Lemma}
\newtheorem{remark}[theorem]{Remark}

 \newcommand{\tcb}{\textcolor{blue}}
 \newcommand{\tcr}{\textcolor{red}}

\begin{document}

\title{On Compatible Transfer Operators in\\Nonsymmetric Algebraic Multigrid}
\author{Ben S. Southworth and Thomas A. Manteuffel}
\maketitle

\begin{abstract}
The standard goal for an effective algebraic multigrid (AMG) algorithm is to develop relaxation and
coarse-grid correction schemes that attenuate complementary error modes.
In the nonsymmetric setting, coarse-grid correction $\Pi$ will almost certainly be
nonorthogonal (and divergent) in any known standard product, meaning $\|\Pi\| > 1$.
This introduces a new consideration,
that one wants coarse-grid correction to be as close to orthogonal as
possible, in an appropriate norm. In addition, due to non-orthogonality,
$\Pi$ may actually amplify certain error modes that are in the range of interpolation.
Relaxation must then not only be complementary to interpolation, but also rapidly eliminate
any error amplified by the non-orthogonal correction, or the algorithm may diverge.
This paper develops analytic formulae on how to construct
``compatible'' transfer operators in nonsymmetric AMG such that $\|\Pi\| = 1$
in some standard matrix-induced norm. Discussion is provided on different
options for the norm in the nonsymmetric setting, the relation between ``ideal''
transfer operators in different norms, and insight into the convergence
of nonsymmetric reduction-based AMG.
\end{abstract}
\section{Background}\label{sec:intro}

\subsection{Algebraic multigrid}

Algebraic multigrid (AMG) is a fixed-point iterative method to solve large
sparse linear systems $A\mathbf{x}=\mathbf{b}$, based on two-parts, relaxation
and coarse-grid correction. These two parts are designed to attenuate
complementary error modes, together resulting in rapid convergence. Relaxation
takes the form of a standard fixed-point iteration,
\begin{align*}
\mathbf{x}_{k+1} = \mathbf{x}_k + Q^{-1}(\mathbf{b} - A\mathbf{x}_k),
\end{align*}
where $Q^{-1}$ is some relatively easy to compute approximation to $A^{-1}$.
Coarse-grid correction is a  subspace correction defined by interpolation and
restriction operators, $P$ and $R$, respectively,
\begin{align*}
\mathbf{x}_{k+1} = \mathbf{x}_k + P(R^*AP)^{-1}R^*(\mathbf{b} - A\mathbf{x}_k).
\end{align*}
Here, $R^*$ restricts the residual to a subspace, a surrogate coarse-grid
operator, $R^*AP$, is inverted in the subspace, and $P$ interpolates the
coarse-grid solution as a correction on the fine grid. If the
coarse-grid operator, $\mathcal{K} := R^*AP$, is too big to invert directly, the
algorithm is called recursively and AMG is applied to (approximately) solve
$(R^*AP)\mathbf{x}_c = R^*\mathbf{r}$.

For presentation purposes, we assume that degrees of freedom (DOFs) can be
partitioned into C-points and F-points and we order the F-points first, followed by 
the C-points. We assume that $R$ and $P$ have an identity over the C-point block, 
that is,
\begin{align} 
A = \begin{bmatrix} A_{ff} & A_{fc} \\ A_{cf} & A_{cc} \end{bmatrix},\hspace{3ex}
	R = \begin{bmatrix} Z \\ I \end{bmatrix}, \hspace{3ex}
	P = \begin{bmatrix} W \\ I \end{bmatrix}.\label{eq:blocks}
\end{align}
As discussed later, this is not technically necessary, but makes the analysis
more tractable and practical. In analyzing AMG methods, one typically
considers bounding error propagation in some norm. Error propagation of
relaxation and coarse-grid correction, respectively, take the forms
\begin{align*}
\mathbf{e} \mapsfrom \hspace{1ex}& (I - Q^{-1}A)\mathbf{e}, \\
\mathbf{e} \mapsfrom \hspace{1ex}& (I - P(R^*AP)^{-1}R^*A)\mathbf{e} \\
	\coloneqq & (I - \Pi)\mathbf{e},
\end{align*}
where $Q\approx A$ is easy to invert, and $\Pi$ denotes the projection
corresponding to coarse-grid correction.

\subsection{Nonsymmetric algebraic multigrid}

For symmetric positive definite (SPD) matrices, letting $R\coloneqq P$ implies $\|\Pi\|_A
= 1$. In the nonsymmetric setting, the $A$-norm {may not be} well-defined. 
Let $A_{sym}= (A+A^*)/2$. Then
$$
\| x \|_A^2 := \la A x , x\ra = \la A_{sym} x, x \ra = \| x \|_{A_{sym}}^2
$$
is a norm over $\Re^n$ only if $A_{sym}$ is SPD. Moreover,  even if $A_{sym}$ is SPD, 
$\Pi$ is not orthogonal in the associated inner product $\la A_{sym} \cdot,\cdot \ra$, 
which implies $\|\Pi\|_{A_{sym}} > 1$.
Let $M$ be an SPD matrix and define $\| x\|_M^2 = \la M x,x\ra$ to be the $M$-norm and 
$\la M x,y\ra$ the associated inner product. We refer to the property
 $\|\Pi\|_M \sim \mathcal{O}(1)$, independent of problem size or
mesh spacing, as being an ``$M$-stable'' coarse-grid correction
\cite{Brezina:2010dm,nonsymm}.

The standard goal for an effective AMG algorithm is to develop relaxation and
coarse-grid correction schemes that attenuate complementary error modes. In the SPD setting, theory defining and guaranteeing convergence (e.g. see \cite{vassilevski2008multilevel,brannick2018optimal}) provide natural measures to target in designing practical AMG methods. In particular \emph{approximation properties} on the interpolation operator $P$ with respect to a chosen relaxation scheme can guarantee convergence, and even provide tight convergence bounds \cite{Falgout:2005hm,brannick2018optimal,vassilevski2008multilevel}. For a thorough review of two-grid algebraic convergence theory for SPD matrices, we refer the reader to \cite[Ch. 3]{vassilevski2008multilevel}.

Approximation properties can also be extended to the nonsymmetric setting, and have been generalized to a so-called ``fractional approximation property'' in \cite{nonsymm}. A number of works have considered nonsymmetric AMG algorithms and Petrov-Galerkin coarse-grid operators, where $R\neq P$. Much of this work has been aggregation-based \cite{Brezina:2010dm,Manteuffel:2017,Notay:2000vy,Sala:2008cv,Wiesner:2014cy,Notay:2010em, notay2018}, but recently several classical CF-AMG variants have been proposed for nonsymmetric problems as well \cite{air1,air2,Manteuffel:2017,Lottes:2017jz}. Theory or heuristics are often used to motivate what properties a ``good'' restriction and interpolation operator should have, typically leaning on the SPD-AMG target of satisfying approximation properties, and then the two transfer operators are constructed largely independently. Unfortunately, approximation properties alone are \emph{not} sufficient to guarantee even two-grid convergence for nonsymmetric AMG \cite{nonsymm,air2}.

In the nonsymmetric setting there is a more fundamental question to AMG method design of whether relaxation and coarse-grid correction are even convergent in some reasonable norm. This introduces a new consideration, that one wants coarse-grid correction to be as close to orthogonal as possible, in an appropriate norm, and, moreover, emphasizes that relaxation and coarse-grid correction {must} be complementary in a more nuanced and challenging sense than in the SPD setting. Traditionally, relaxation is effective at attenuating error \textit{not} in the range of interpolation. For nonsymmetric problems, relaxation must also quickly attenuate error that is amplified by the non-orthogonal coarse-grid correction. In particular, $\Pi$ may amplify certain error modes that are \textit{in the range of interpolation}, and relaxation must eliminate this error rapidly, or the algorithm may diverge. In our opinion, constructing $R$ and $P$ so that $\Pi$ is almost orthogonal is the most viable approach to addressing this problem, in part because AMG almost always relies on fairly simple relaxation operators (in contrast to, e.g., geometric MG). 

To our knowledge, this necessary relationship between relaxation and coarse-grid correction for nonsymmetric AMG has not been discussed in the literature, and there is minimal theory on how to go about constructing $R$ and $P$ on a practical level to ensure almost orthogonal ``stable'' projections. Nonsymmetric AMG based on approximate ideal restriction (AIR) \cite{air1,air2} is probably the closest method to considering compatible $R$ and $P$. There, some compatibility of $R$ and $P$ is taken into consideration in the framework and algorithm design, but not directly during the construction of operators. It was shown in \cite{nonsymm} that approximation properties on $R$ and $P$, a stable coarse-grid correction, and sufficient relaxation to damp divergent coarse-grid correction modes, are sufficient conditions for two-grid and multigrid W-cycle convergence. In addition, stability of coarse-grid correction is a necessary condition for two-grid convergence \cite{nonsymm}. Algorithms targeting approximation properties for $R$ and $P$ can often be extended from the SPD setting, e.g., \cite{Manteuffel:2017,Sala:2008cv,Brezina:2010dm}, and AMG in general relies on standard algebraic relaxation schemes (of which the stronger and more expensive variants may be needed in nonsymmetric AMG). However, to the best of our knowledge, at no point has theory been developed that provides a framework and target for constructing $R$ and $P$ \textit{together} such that they are, in some sense, compatible, and yield an $M$-stable coarse-grid correction. 

\subsection{Contributions}

In this paper we focus on how to construct $R$ and $P$ of the form \eqref{eq:blocks} such that the resulting $\Pi$ is $M$-stable in some meaningful or appropriate $M$-norm. We use the terms meaningful and appropriate to imply the existence of a
relaxation that is effective on the $M$-orthogonal complement of the range of $\Pi$
(range of $P$) and also attenuates any errors in the range of $P$ that are amplified
by $\Pi$.

It has been recognized that for difficult SPD problems, the choice and
construction of coarse grids, relaxation, and interpolation must be done in
conjunction, leading to the development of compatible relaxation techniques
\cite{Brannick.Falgout.2010,Livne.Livne.2004}.
Here, we develop a similar framework for compatible transfer
operators in nonsymmetric AMG. We begin by developing background
theory on projections in \Cref{sec:proj}. Given an $M$-induced inner product
for SPD matrix $M$, and restriction or
interpolation operator, we refer to a compatible restriction/interpolation
operator with respect to $M$ as one such that 
$\|\Pi\|_M = \|P(R^*AP)^{-1}R^*A\|_M = 1$ in the induced norm.
\Cref{sec:orth} develops necessary and sufficient conditions on $R$ and $P$
such that $\|\Pi\|_M = 1$ in any $M$-induced matrix norm, followed by closed forms for compatible interpolation
and restriction operators for the common choices of $M$, including $M = I$, $M=\sqrt{A^*A}$,  and $M = A^*A$, and a relation between the ``ideal'' multigrid transfer operators in different norms. Further observations are provided in \Cref{sec:obs}, including observations on the compatibility of relaxation with non-orthogonal projections in AIR. We do not address approximation properties directly in this work, but believe the newly developed theory coupled with care given to approximation properties and potentially relaxation will facilitate algorithm design for faster and more robust nonsymmetric AMG \cite{nonsymm}.

\section{Linear algebra and projections}\label{sec:proj}

Given any finite dimensional linear operator, we have the following. Let $A$ be an $m\times n$ matrix and define the following notation
\begin{alignat*}{2}
\CR(A) &= \mbox{range of $A$}, \hspace{10ex}
	&A^* &= \mbox{adjoint in standard inner product}, \\
\CK(A) &= \mbox{null space of $A$},
	&\dim(\CR(A))&= \mbox{dimension of $\CR(A)$}.
\end{alignat*}
The fundamental theorem of linear algebra yields:
\begin{alignat*}{2}
\dim(\CR(A)) &= \dim(\CR(A^*)), \hspace{5ex}
	&\CR(A) &\perp \CK(A^*),\\
\dim(\CK(A)) &= \dim(\CK(A^*)), 
	&\CR(A^*) &\perp \CK(A).
\end{alignat*}
Let $M$ be SPD and define the $M$-inner product and associated notation
\begin{alignat*}{3}
\la x, y \ra_M &= \la M x, y  \ra,\hspace{3ex}
	&\| x \|_M^2 &= \la M x, x \ra, \hspace{3ex}
 	&x \perp_M y &\Leftrightarrow \la x, y \ra_M =0, \\
\la A x, y \ra_M &= \la x, A^\dagger y  \ra_M,
&A^\dagger &= M^{-1} A^* M.
\end{alignat*}
For any subspace $\CV$, let
\begin{displaymath}
\CV^{\perp_M} = \{ x ~:~ \la x, y \ra_M =0  ~~\forall y \in \CV\}.
\end{displaymath}
It is clear that $\CV^{\perp_M} = M^{-1} \CV^\perp$, and we have the following generalization:
\begin{equation*}
\CR(A) \perp_M \CK(A^\dagger),\hspace{5ex}
\CR(A^\dagger) \perp_M \CK(A).
\end{equation*}

\subsection{Projections}

The operator $\Pi$ is a projection if $\Pi^2 = \Pi$. Then, $(I-\Pi)$ is also a projection and
\begin{equation*}
\CR(I-\Pi) = \CK(\Pi),\hspace{5ex}
\CK(I-\Pi) = \CR(\Pi).
\end{equation*} 
For any SPD $M$ we have
\begin{equation*}
\| \Pi \|_M = \sup_{x\neq 0} \frac{ \| \Pi x\|_M}{\| x \|_M},\hspace{5ex}
\|\Pi\|_M = \| I-\Pi \|_M.
\end{equation*}
Consider the following representation of $\Pi$. Let $n_c = $ dim$(\CR(\Pi))$. Construct $n\times n_c$ matrices $V$ and $U$ such that
\begin{align*}
 \CR(\Pi)  &= \CR(V), \hspace{3ex}
 	\CK(\Pi)^{\perp_M} = \CR(U) = \CR(\Pi^\dagger), \hspace{3ex}
	V^*M V = U^*M U = I.
\end{align*}
Then,
\begin{align*}
\Pi &= V(U^\dagger V)^{-1}U^\dagger
	= V(U^*MV)^{-1} U^*M, \hspace{5ex}
\|\Pi \|_M= \|(U^*MV)^{-1} \|.
\end{align*}
This last line takes a bit of work, but it is straightforward.

Next, let $B_R$, $B_P$ be any $n_c\times n_c$ nonsingular matrices. Define
\begin{equation*}
R = UB_R ,\hspace{5ex}
P = VB_P.
\end{equation*}
Then, we can also write
\begin{align*}
	\Pi = P(R^*MP)^{-1} R^*M, \hspace{5ex}
	\| \Pi \|_M &= \| B_P (R^*MP)^{-1}B_R^*\|.
\end{align*}

\subsection{Orthogonal projections}
We first show that a projection is orthogonal in the $M$-inner product if $\Pi = \Pi^\dagger$. In this case,
\begin{displaymath}
\Pi x \perp_M (I-\Pi)x,
\end{displaymath}
and 
\begin{displaymath}
\| \Pi \|_M = 1.0.
\end{displaymath}
To see this, write
\begin{equation*}
\la \Pi x, (I-\Pi) x\ra_M = \la x, \Pi^\dagger (I-\Pi)x\ra_M = \la x, \Pi(I-\Pi)x \ra_M = 0,
\end{equation*}
and, for every $x$,
\begin{equation*}
\| x \|_M^2 = \| \Pi x + (I-\Pi)x \|_M^2 = \|\Pi x\|_M^2 + \|(I-\Pi)x\|_M^2 \geq \|\Pi x\|_M^2.
\end{equation*}
For $x\in\CR(\Pi)$, we have $\|x\|_M = \| \Pi x\|$, which yields the result. We are now in a position to prove two simple lemmas.

\begin{lemma}\label{lem:orthog}
A projection, $\Pi$, is orthogonal in the M-inner product if and only if 
\begin{equation}
\CR(\Pi) \perp_M \CK(\Pi)
\end{equation}
\end{lemma}

\begin{proof}
Assume $\Pi = \Pi^\dagger$. Recall that $\CK(\Pi) = \CR(I-\Pi)$. If $\Pi = \Pi^\dagger$, then, for every $x,y$,
\begin{displaymath}
\la \Pi x, (I-\Pi)y \ra_M =  \la  x, \Pi^\dagger (I-\Pi)y \ra_M = \la  x, \Pi(I-\Pi)y \ra_M  = 0.
\end{displaymath}
Thus, $\CR(\Pi) \perp_M \CK(\Pi)$.

Now, assume $\CR(\Pi) \perp_M \CK(\Pi)$. For every $x,y$, 
\begin{displaymath}
0= \la \Pi x, (I-\Pi)y \ra_M =  \la  x, \Pi^\dagger (I-\Pi)y \ra_M  = \la  x, \Pi^\dagger-\Pi^\dagger\Pi)y \ra_M ,
\end{displaymath}
which implies, for every $y$,
\begin{displaymath}
\Pi^\dagger y = \Pi^\dagger \Pi y
\end{displaymath}
If $y \in \CK(\Pi)$, then $\Pi^\dagger y = 0$.  Thus, $\CK(\Pi) \subset\CK(\Pi^\dagger)$. By a dimensionality argument, 
$\CK(\Pi) =\CK(\Pi^\dagger)$. Similar to above, for every $x,y$,
\begin{displaymath}
0= \la \Pi x, (I-\Pi)y \ra_M = \la (I-\Pi^\dagger)\Pi x, y \ra_M = \la (\Pi -\Pi^\dagger\Pi) x, y \ra_M. 
\end{displaymath}
This implies, for every $x$,
\begin{displaymath}
\Pi x = \Pi^\dagger \Pi x.
\end{displaymath}
Let $x \in \CR(\Pi)$. Then, $x = \Pi^\dagger x$. Thus, $\CR(\Pi) \subset \CR(\Pi^\dagger)$. A dimensionality argument yields $\CR(\Pi)=\CR(\Pi^\dagger)$. Thus, $\Pi = \Pi^\dagger$, which completes the proof.
\end{proof}

\begin{corollary}\label{cor:orthogonal}
The following are equivalent, all corresponding to an $M$-orthogonal projection:
\begin{displaymath}
\begin{array}{ccc}
\begin{array}{rcl}
\Pi &=& \Pi^\dagger \\
\Pi &=& M^{-1} \Pi^* M \\
M\Pi &=& \Pi^* M \\
M\Pi &=& (M\Pi)^*\\~~
\end{array}
&\qquad&
\begin{array}{rcl}
\CR(\Pi) &\perp_M& \CK(\Pi) = \CR(\Pi^*)^\perp\\
\CR(M\Pi) &\perp & \CR(\Pi^*)^\perp \\
\CR(M\Pi) &=& \CR(\Pi^*) \\
\CR(\Pi) &\perp_M& \CR(I-\Pi) \\
\CR(M\Pi) &\perp& \CR(I-\Pi)
\end{array}
\end{array}
\end{displaymath}
\end{corollary}
The conditions that will be of most use later in this paper are $\CR(\Pi) \perp_M \CK(\Pi)$ and  $\CR(M\Pi) = \CR(\Pi^*)$.

\subsection{Non-orthogonal projections}

Projections are geometric in nature. To that end, define the minimal canonical
angle, $\theta_{min}$, between subspaces $\mathcal{X},\mathcal{Y}$, as
\begin{align}
\cos\left(\theta_{min}^{[\mathcal{X},\mathcal{Y}]_M}\right) & = \sup_{\substack{\mathbf{x}\in\mathcal{X},\\ \mathbf{y}\in\mathcal{Y} } }
	\frac{|\langle \mathbf{x},\mathbf{y}\rangle_M|}{\|\mathbf{x}\|_M\|\mathbf{y}\|_M}, \label{eq:theta1}
\end{align}
for SPD $M$ and inner product, $\langle\cdot,\cdot\rangle_M$. Note that $\theta_{min}^{[\mathcal{X},\mathcal{Y}]_M}
= \theta_{min}^{[\mathcal{X}^{\perp_M},\mathcal{Y}^{\perp_M}]_M}$ \cite{Deutsch:1995tc,Szyld:2006bg} and
\begin{align}
\|\Pi^\dagger\|_M^2 = \|\Pi\|_M^2 = \frac{1}{\sin^2\left(\theta_{min}^{[\mathcal{R}(\Pi),\mathcal{K}(\Pi)]_M}\right)}
	= 1 + \cot^2\left(\theta_{min}^{[\mathcal{R}(\Pi),\mathcal{K}(\Pi)]_M}\right).\label{eq:cot}
\end{align}

In the case of an {$M$-}orthogonal projection, $\theta_{min}^{[\mathcal{R}(\Pi),\mathcal{K}(\Pi)]_M} = \pi/2$ and $\cot(\theta_{min}^{[\mathcal{R}(\Pi),\mathcal{K}(\Pi)]_M}) = 0$. For non-orthogonal projections, the sup over $\mathbf{x}\in\mathcal{R}(\Pi)^{\perp_M}$
can be seen as a measure of the non-orthogonality.

\begin{lemma}\label{th:M}
For SPD matrix $M$ and projection $\Pi$,
\begin{displaymath}
\| \Pi \|_M^2 = 1 + \sup_{x\in \CK(\Pi^\dagger)}\frac{\|\Pi x\|_M^2}{\| x\|_M^2} 
= 1 + \sup_{x\in \CR(\Pi)^{\perp_M}}\frac{\|\Pi x\|_M^2}{\| x\|_M^2}
\end{displaymath}
Furthermore, $\sup_{\mathbf{x}\in\mathcal{R}(\Pi)^{\perp_M}} \frac{\|\Pi\mathbf{x}\|_M}{\|\mathbf{x}\|_M} = \cot\left(\theta_{min}^{[\mathcal{R}(\Pi),\mathcal{K}(\Pi)]_M}\right)$.
\end{lemma}
\begin{proof}
Recall, $\|\Pi\|_M = \| (I-\Pi)\|_M$. Given $x$, let $x = x_1 + x_2$ where $x_1 \in \CR(\Pi) = \CK(I-\Pi)$ and
$x_2 \in \CR(\Pi)^{\perp_M} = \CK(\Pi^\dagger)$. Note that $x_1\perp_M x_2$. Then,
\begin{eqnarray*}
\| I-\Pi \|_M^2 &=& \sup_{x} \frac{\la  (I-\Pi)x, (I-\Pi) x\ra_M}{\la x,x \ra_M} 
=  \sup_{x} \frac{\la  (I-\Pi)x_2, (I-\Pi) x_2\ra_M}{\la x_1,x_1 \ra_M + \la x_2,x_2 \ra_M} \\
&=& \sup_{x} \frac{\|x_2\|_M^2 - 2\la  \Pi x_2, x_2\ra_M+ \|\Pi x_2\|_M^2}{\| x_1\|_M^2 + \| x_2\|_M^2} \\
&=& \sup_{x} \frac{\|x_2\|_M^2 + \|\Pi x_2\|_M^2}{\| x_1\|_M^2 + \| x_2\|_M^2} 
= 1 + \sup_{x_2 \in \CR(\Pi)^{\perp_M}} \frac{\|\Pi x_2 \|_m^2}{\|x_2\|_M^2}.
\end{eqnarray*} 
The relation to the cotangent of the minimum angle between $\mathcal{R}(\Pi)$ and $\mathcal{K}(\Pi)$ follows from \eqref{eq:cot}. This completes the proof.
\end{proof}

Note that this result can be seen as an extension of \cref{lem:orthog}; if $\CR(\Pi) \perp_M \CK(\Pi)$, then obviously $\sup_{x\in \CR(\Pi)^{\perp_M}}\frac{\|\Pi x\|_M^2}{\| x\|_M^2} = 0$, and from \cref{th:M} $\|\Pi\|_M^2= 1$. \cref{th:M} also provides a potential optimization objective to form (almost) orthogonal transfer operators in practice. The following sections explore representations of $\mathcal{R}(\Pi)$ and its $M$-orthogonal complement to derive exactly orthogonal transfer operators. In the context of \cref{th:M}, these relations can also be used to measure and/or minimize the non-orthogonal component of $\Pi$.

\section{Compatible transfer operators in AMG}\label{sec:orth}

Let $A$ be an $n\times n$, nonsingular, possibly nonsymmetric, matrix.  
Let $R$ and $P$ be $n\times n_c$ interpolation and restriction matrices. Consider the
projection corresponding to coarse-grid correction,
\begin{equation}\label{eq:pi}
\Pi = P(R^*AP)^{-1} R^*A.
\end{equation}
If $A$ is SPD and $R = P$, then $\Pi$ is orthogonal in the A-inner product and, thus, $\|\Pi\|_A = 1.0$. 
If $A$ is nonsymmetric, then $A$ does not yield an inner product, and it is unclear which norm to measure this in. 
In this section, we examine how to construct ``compatible'' transfer operators $R$ and $P$ such that resulting $\Pi$ 
is $M$-orthogonal for  various $M$. 

To begin, we take the CF-splitting approach to AMG, where interpolation and restriction are defined as in \eqref{eq:blocks}. We want to define F-point 
blocks $Z$ and $W$ such that $\Pi$ is orthogonal in the $M$-inner product. Appealing to \Cref{lem:orthog} and \Cref{cor:orthogonal}, this is equivalent to satisfying $\CR(\Pi) \perp_M \CK(\Pi)$. With $\Pi$ defined as in \eqref{eq:pi}, $\CK(\Pi) = \CR(A^*R)^\perp$. We then have equivalent conditions
\begin{align}
	\CR(\Pi) \perp_M \CR(A^*R)^\perp \hspace{3ex}
		&\Longleftrightarrow\hspace{3ex} 
		M\CR(\Pi) \perp \CR(A^*R)^\perp \\
	&\Longleftrightarrow\hspace{3ex} M\CR(P)= \CR(A^*R).
\end{align}
We can formalize this in the following lemma.

\begin{lemma}\label{lem:ortho1}
Let $\Pi= P(R^*AP)^{-1} R^*A$. Then, $\Pi$ is $M$-orthogonal if we can find $n_c\times n_c$, nonsingular matrices, $B_P$ and $B_R$, such that
\begin{equation}\label{eq:ortho1}
MPB_P =   A^*RB_R.
\end{equation}
\end{lemma}

In the case of $R$ and $P$ in the CF-splitting form of AMG \eqref{eq:blocks}, \eqref{eq:ortho1} corresponds to: 
given $M$, find $\{Z,W\}$ and $\{B_R, B_P\}$ such that
\begin{displaymath}
M \left[ \begin{array}{c} W \\ I  \end{array}\right] B_P = A^* \left[ \begin{array}{c} Z \\ I  \end{array}\right] B_R.
\end{displaymath}
Note, in principle we only need $B_P$ \emph{or} $B_R$, but depending on $M$, working with one may be more convenient than the other. The following subsections we consider several natural examples of $M=I$, $M= A^*A$, $M= (A^*A)^{1/2}$, $M = A_{sym}$, and $M = A^*A_{sym}^{-1}A$.  

\subsection{$M=I$} 
Appealing to \Cref{lem:ortho1} we seek $B_P$ and $B_R$ such that $PB_P = A^*RB_R$. Expanding in block form yields
\begin{equation*}
 \left[ \begin{array}{c} W \\ I  \end{array}\right] B_P =  \left[ \begin{array}{c} A_{ff}^*Z+ A_{cf}^* \\ A_{fc}^*Z+ A_{cc}^*  \end{array}\right]B_R,
\end{equation*}
This can be accomplished by setting $B_P = (A_{fc}^*Z+ A_{cc}^*)$ and $B_R = I$ to get the condition
\begin{equation*}
WA_{fc}^*Z+ WA_{cc}^* - A_{ff}^*Z- A_{cf}^* = \mathbf{0}.
\end{equation*}
Taking the adjoint to work with non-transpose submatrices of $A$, and rearranging to solve for $W$ given $Z$ or $Z$ given $W$, yields the two (equivalent) conditions on $W$ and $Z$ such that $\Pi$ is $I$-orthogonal:
\begin{subequations}\label{eq:M=I}
\begin{align}
(Z^*A_{fc}+ A_{cc})W^* &= Z^*A_{ff}+ A_{cf}, \\
Z^*(A_{ff} - A_{fc}W^*) &= A_{cc}W^* -A_{cf}.
\end{align}
\end{subequations}
Note that the case of classical ideal restriction, $Z^*= -A_{cf}A_{ff}^{-1}$ \cite{air1,air2}, yields $W=\mathbf{0}$.

\subsection{$M=A^*A$}

Appealing to \Cref{lem:ortho1} we seek $B_R$ and $B_P$ such that $A^*APB_P = A^*RB_R\Leftrightarrow APB_P = RB_R$. Expanding in block form yields
\begin{equation*}
 \left[ \begin{array}{c} A_{ff}W+A_{fc} \\ A_{cf}W + A_{cc}  \end{array}\right] B_P = \left[ \begin{array}{c} Z \\ I  \end{array}\right]B_R,
\end{equation*}
In this case, choose $B_P = I$ and $B_R = (A_{cf}W + A_{cc})$ to arrive at the condition
\begin{equation*}
	ZA_{cf}W + ZA_{cc} - A_{ff}W - A_{fc} = \mathbf{0}.
\end{equation*}
Rearranging to solve for $W$ given $Z$ or $Z$ given $W$, yields the two (equivalent) conditions on $W$ and $Z$ such that $\Pi$ is $A^*A$-orthogonal,
\begin{subequations}\label{eq:M=A*A}
\begin{align}
 (A_{ff}- ZA_{cf} ) W &= ZA_{cc} -A_{fc}, \\
 Z(A_{cf}W + A_{cc}) & = A_{ff}W+A_{fc}.
\end{align}
\end{subequations}
These equations have a similar structure to the case of $M=I$ \eqref{eq:M=I}. Here, note that $Z=\mathbf{0}$ yields $W= -A_{ff}^{-1}A_{fc}$, which is ideal interpolation.

\subsection{$M=(A^*A)^{1/2}$}

For completeness we consider the norm of $M=(A^*A)^{1/2}$, tying into our theory paper on convergence of nonsymmetric AMG \cite{nonsymm}. Let $A$ have singular value decomposition $A= U\Sigma V^*$, where $U$ and $V$ are unitary and the diagonal matrix $\Sigma$ contains the singular values. Note that
\begin{displaymath}
 M = (A^*A)^{1/2} = V\Sigma V^*.
\end{displaymath}
Appealing to \Cref{lem:ortho1},  we seek $V\Sigma V^*PB_P = V\Sigma U^*RB_R \Leftrightarrow V^*PB_P= U^*RB_R$. In \cite{nonsymm}, we compared $V^* P$ and $U^*R$ and proved that if they have the same range, then $\Pi$ is $(A^*A)^{1/2}$-orthogonal, consistent with the framework developed herein.

\subsection{$M=A_{sym} = (A+A^*)/2$}  \label{sub:M=As}
If $A_{sym}$ is SPD, then this can be used to define a norm.
Appealing to (\ref{eq:ortho1}), $\Pi$ will be $M$-orthogonal if
there exist $P$ and $R$, and scaling matrices $B_P$ and $B_R$ such that
\begin{align}\label{eq:M=As}
A_{sym}PB_P = A^*RB_R.
\end{align}
Following the development above, given $Z$, set $B_P = I$ and one can solve for $W$ and $B_R$ that satisfies (\ref{eq:M=As}).
Likewise, given $W$, set $B_R = I$ and  one may solve for $Z$ and $B_P$. The resulting formula are complicated and we choose
to omit them. However, it will be shown below that a specific  choice of $P$ and $R$ satisfies (\ref{eq:M=As}).

\subsection{$M = A^* A_{sym}^{-1} A$}  Again, if $A_{sym}$ is SPD, then $M$ can be used to define a norm.  
It is an interesting norm because, in the case $A$  is a discrete form of second-order elliptic equation,  it  
has properties similar to $(A^*A)^{1/2}$. Appealing again to (\ref{eq:ortho1})
yields
\begin{align}
\nonumber
A^*A_{sym}^{-1}APB_P &= A^*RB_R ,\\ \label{eq:M=A*AsA}
APB_P &= A_{sym} R B_R.
\end{align}
Compare (\ref{eq:M=A*AsA}) with (\ref{eq:M=As}). As in section \ref{sub:M=As}, the formulas are 
complicated. However, it will be shown in the next section that a simple choice for $P$ and $R$ 
can be made to satisfy (\ref{eq:M=A*AsA}).

\subsection{Pairs of Ideal Transfer Operators}\label{sec:orth:comp}

For a given CF-spitting, the so-called ideal transfer operators can be defined on any matrix and there is a certain symmetry to the ideal $R$ and $P$. Consider a generic, nonsingular matrix, Q, in CF-block form,
$$
Q = \left[ \begin{array}{cc}
Q_{ff} & Q_{fc} \\ Q_{cf} & Q_{cc} 
\end{array} \right].
$$
If $Q_{ff}$ is nonsingular, we define ideal transfer operators of $Q$ to have the form
$$
P_{ideal}(Q) =  \left[\begin{array}{c} W_{ideal}(Q) \\ I \end{array}\right],
\quad R_{ideal}(Q) =  \left[\begin{array}{c} Z_{ideal}(Q) \\ I \end{array}\right] ,
$$
where $W_{ideal}(Q)$ and $Z_{ideal}(Q)$ are such that
\begin{align}
QP_{ideal}(Q) &= 
\left[ \begin{array}{cc}
Q_{ff} & Q_{fc} \\ Q_{fc} & Q_{cc} 
\end{array} \right] 
\left[\begin{array}{c} W_{ideal}(Q) \\ I \end{array}\right] 
=
\left[\begin{array}{c} 0 \\ \CS_Q \end{array}\right] , \\
R_{ideal}(Q)^*Q &= \left[\begin{array}{cc} Z_{ideal}(Q)^* &I \end{array}\right]
\left[ \begin{array}{cc}
Q_{ff} & Q_{fc} \\ Q_{cf} & Q_{cc} 
\end{array} \right]  
=
\left[\begin{array}{cc} 0 & \CS_Q \end{array}\right], 
\end{align}
and $\CS_Q = Q_{cc} - Q_{cf}Q_{ff}^{-1}Q_{fc}$ is the Schur complement of $Q$ \cite{air2}. Clearly, 
\begin{align*}
W_{ideal}(Q) &= - Q_{ff}^{-1}Q_{fc},\\
Z_{ideal}(Q)^* & =- Q_{cf} Q_{ff}^{-1}.
\end{align*}
{If we suppose that $Q = A$ and $P = P_{ideal}(Q)$, then for any choice of $R$ the associated coarse grid matrix is}
$$
A_c =  R^* A P = S_Q,
$$
as discussed in e.g., \cite{air1,air2}. However, the resulting $\Pi = P (R^*AP)^{-1} R^*A$ will be $M$-orthogonal if and only if (\ref{eq:ortho1}) is satisfied. 

More generally, given an SPD matrix $M$ and arbitrary nonsingular matrix $Q$, (\ref{eq:ortho1}) describes how to choose pairs $P$ and $R$ such that the resulting 
$\Pi$ is $M$-orthogonal. Rearranging (\ref{eq:ortho1}), setting $B_P = I$, and multiplying by $Q$ yields
\begin{align}
P &= M^{-1} A^* RB_R, \\
QP &=  QM^{-1} A^* R B_R.
\end{align}
If $P = P_{ideal}(Q)$, then
\begin{align}
QP = \left[ \begin{array}{c} 0\\ S_Q\end{array}\right] = QM^{-1} A^* R B_R,
\end{align}
which implies
\begin{align}\label{eq:RfromP}
R = R_{ideal}(AM^{-1}Q^*).
\end{align}
Similarly, rearranging (\ref{eq:ortho1}), setting $B_R = I$, and multiplying through by $Q^*$ yields
\begin{align}
A^{-*}M PB_P &= RB_R, \\
Q^*A^{-*}MPB_P &  = Q^*R.
\end{align}
If $R=R_{ideal} (Q)$, then 
\begin{align}\label{eq:PfromR}
Q^*A^{-*}MPB_P &= Q^*R = \left[\begin{array}{c} 0 \\S_Q^*\end{array}\right],
\end{align}
which implies $P = P_{ideal}(Q^*A^{-*}M)$. The above discussion leads to the following lemma.

\begin{lemma}\label{lem:ideal_pairs}  Let $M$ be an SPD matrix and define $\Pi = P(R^*AP)^{-1}R^*A$.  
If $P = P_{ideal}(Q)$, then $\Pi$ will be $M$-orthogonal if and only if $R=R_{ideal}(AM^{-1}Q^*)$. 
Likewise, if $R = R_{ideal}(Q)$, then $\Pi$ will be $M$-orthogonal if and only if $P=P_{ideal}(Q^* A^{-*}M)$.
\end{lemma}
\begin{proof}
The proof follows from the discussion above.
\end{proof}

Next, various choices for $M$ and $Q$ will be examined. For the choices of $M$ described above,
and a set of choices for $Q$ such that $P_{ideal}(Q)$ is computable, (or at least can be approximated), the choice of $R$ that makes $\Pi$ $M$-orthogonal will be described, and similarly for $R_{ideal}(Q)$ and $P$ such that $\Pi$ is $M$-orthogonal. Results are shown in \Cref{tab:P,tab:R}.

\begin{table}[!htb]
\centering
$
\begin{array}{|c|ccc|cc|} \hline
M \text{ $\backslash$ } Q & I & A & A_{sym} & A^*A & AA^* \\ \hline
I & \tcr{A} & \tcb{AA^*} & \tcb{AA_{sym}} & AA^*A & A^2A^* \\
A & \tcr{I} & \tcr{A} & \tcr{A} & \tcb{A^2} & \tcb{A^2} \\
A_{sym}  & AA_{sym}^{-1} & A A_{sym}^{-1}A^* & \tcr{A} & A A_{sym}^{-1} A^*A & A A_{sym}^{-1} AA^* \\ \hline
A^*A & A^{-*} &\tcr{I} & A^{-*}A_{sym} & \tcr{A} & A^{-*}AA^* \\
A^*A_{sym}^{-1}A & A_{sym}A^{-*} &\tcr{A_{sym}} & A^{-*}A_{sym} & \tcb{A_{sym} A} & A_{sym} A^{-*} AA^* \\ \hline
\end{array}
$
\caption{Choose $P=P_{ideal}(Q)$ for matrix $Q$ in the top row. 
For each choice of $M$ in the first column, the table shows $AM^{-1}Q^*$ such that $R=R_{ideal}(AM^{-1}Q^*)$
yields an $M$-orthogonal $\Pi$. For $M=A$ we assume $A$ is SPD and for $M= A_{sym}$ we assume $A_{sym}$ is SPD.}
\label{tab:P}
\end{table}

\begin{table}[!htb]
\centering
$
\begin{array}{|c|ccc|cc|} \hline
M \text{ $\backslash$ } Q & I & A & A_{sym} & A^*A & AA^* \\ \hline
I & A^{-*} & \tcr{I} & A_{sym}^*A^{-*} & A^*AA^{-*} & \tcr{A} \\
A & \tcr{I} & \tcr{A} & \tcr{A} & \tcb{A^2} & \tcb{A^2} \\
A_{sym}  & A^{-*}A_{sym} & \tcr{A_{sym}}& A_{sym}A^{-*}A_{sym} &  A^*A A^{-*} A_{sym}& \tcb{A A_{sym}} \\ \hline
A^*A & \tcr{A} & \tcb{A^*A} & \tcb{A_{sym}A} & A^{*}A^2 & AA^*A \\
A^*A_{sym}^{-1}A & A_{sym}^{-1}A & A^*A_{sym}^{-1}A & \tcr{A} &A^*A A_{sym}^{-1} A & AA^*A_{sym}^{-1} A \\ \hline
\end{array}
$
\caption{Choose $R=R_{ideal}(Q)$ for matrix $Q$ in the top row. For each choice of $M$ in the first column, the table shows $Q^*A^{-*} M$ such that $P=P_{ideal}(Q^*A^{-*}M)$ yields an $M$-orthogonal $\Pi$. For $M=A$ we assume $A$ is SPD and for $M= A_{sym}$ we assume $A_{sym}$ is SPD.}
\label{tab:R}
\end{table}

We remind the reader that $M$ never need be computed, but both $P$ and $R$ do. In \Cref{tab:P,tab:R}, the entries marked in red involve a single operator for both $P$ and $R$, while those marked in blue involve the product of two operators for either $P$ or $R$. The other entries are either not computable or involve three or more operators. {It is difficult to precisely quantify or compare computability of different pairs of operators in practice; most are too expensive to form directly and must be approximated, and thus the cost depends entirely on the approximation algorithm. However, it is likely that algorithmic cost and approximation accuracy both depend highly on the sparsity of the operator being approximated, and to this end we believe the pairs that involve a single operator for $R$ and $P$ (marked in red) are the most likely to be approximated well in a practical algorithm.} Focusing on the red entries, we see the following combinations yield $M$-orthogonal projections:
\begin{enumerate}
\item{$M= I$, $P = P_{ideal}(I)$, $R=R_{ideal}(A)$.} In the original AIR algorithm, $R$ is chosen to be approximately ideal with respect to $A$. However, if $P$ is chosen to be $P_{ideal}(I)$, that is, $W=\mathbf{0}$, ensuring $\|\Pi\|$ is bounded independent of problem size requires that either $R$ approximates $R_{ideal}(A)$ or F-relaxation is applied, with accuracy that depends on mesh spacing $h$. I.e., $R\to R_{ideal}$ as $h\to 0$ (see \cite[Lemma 1]{air2}). To improve stability, a minimal $W\neq \mathbf{0}$ is chosen to interpolate smooth functions, tying into the approximation property discussion in \Cref{sec:intro}. 

\item {$M= A_{sym}$, $P=P_{ideal}(A_{sym})$, $R = R_{ideal}(A)$.} This is a modification of the original AIR algorithm. Here, $R$ is chosen to be ideal with respect to $A$, but $P$ is chosen to be ideal
with respect to $A_{sym}$. The result is that $\Pi$ is $A_{sym}$-orthogonal. Note that this requires $A_{sym}$ to be SPD.
\item{$M=A^*A$, $P=P_{ideal}(A)$, $R=R_{ideal}(I)$}. Here, $P$ is ideal interpolation with respect to $A$, and $R$ selects only the coarse grid equations. The result is an $A^*A$-orthogonal projection. Notice the symmetry with item 1. Here, however, the orthogonality is in the stronger $A^*A$-norm.
\item{$M=A^*A_{sym}^{-1}A$, $P =P_{ideal}(A)$, and $R  = R_{ideal}(A_{sym})$.} This pair is similar to item 2 above. Here, $P$ is chosen to be ideal interpolation with respect to $A$, and $R$ is chosen to be ideal restriction with respect to $A_{sym}$. The result is orthogonality in a stronger norm.
\end{enumerate}

{Which norm to use in practice is a difficult question, and generally problem dependent. For example, AIR \cite{air1,air2} is designed for advection problems, and the transfer operators are chosen to be approximately $\ell^2$-orthogonal, a natural choice for hyperbolic problems. If one were instead considering parabolic problems, $A_{sym}$ may be more natural. In addition, one wants a norm in which orthogonal $R$ and $P$ are tractable to approximate, tying into the discussion above regarding orthogonal pairs consisting of single operators. Finally, although a single pair of $R$ and $P$ may be orthogonal in multiple norms (see \Cref{rem:unique}), ideally one uses a known norm in which we know a-priori the modes over which relaxation is effective at attenuating error, and in which coarse-grid correction must be complementary (i.e., orthogonal in a random or unknown norm may not be useful).}

\section{Observations}\label{sec:obs}

\begin{remark}[Symmetry between $P$ and $R$ pairs]
Building on the general results above, here we demonstrate an interesting symmetry in transfer operators that yield an orthogonal projection, along with some closed forms for ideal operators. Note the identities 
\begin{align*}
AA^* & = \begin{bmatrix} A_{ff}A_{ff}^* + A_{fc}A_{fc}^* & A_{ff}A_{cf}^* + A_{fc}A_{cc}^* \\
  A_{cf}A_{ff}^* + A_{cc}A_{fc}^* & A_{cf}A_{ff}^* + A_{cc}A_{cc}^*\end{bmatrix}, \nonumber\\
A^*A & = \begin{bmatrix} A_{ff}^*A_{ff} + A_{cf}^*A_{cf} & A_{ff}^*A_{fc} + A_{cf}^*A_{cc} \\ 
  A_{fc}^*A_{ff} + A_{cc}^*A_{cf} & A_{fc}^*A_{fc} + A_{cc}^*A_{cc}\end{bmatrix}, \nonumber\\
A^{-1} & = \begin{bmatrix} \mathcal{S}_{F}^{-1} & -\mathcal{S}_{F}^{-1}A_{fc}A_{cc}^{-1} \\
  -A_{cc}^{-1}A_{cf}\mathcal{S}_{F}^{-1} & \mathcal{S}_{C}^{-1} \end{bmatrix},
\end{align*}
where $\mathcal{S}_{C} := A_{cc} - A_{cf}A_{ff}^{-1}A_{fc}$ and
$\mathcal{S}_{F} := A_{ff} - A_{fc}A_{cc}^{-1}A_{cf}$ denote the two Schur
complements of $A$. 

Then, a few of the ideal formulae of interest are:
\begin{align}
\Big[Z_{\textnormal{ideal}}{(A^{-*})}\Big]^* & = A_{cc}^{-*}A_{fc}^* , \hspace{12.25ex} W_{\textnormal{ideal}}{(A^{-*})} = A_{cf}^*A_{cc}^{-*}, \label{eq:ideal-Ainv}\\
\Big[Z_{\textnormal{ideal}}{(I)}\Big]^* & =\mathbf{0} , \hspace{20.25ex} W_{\textnormal{ideal}}{(I)} = \mathbf{0}, \nonumber\\
\Big[Z_{\textnormal{ideal}}{(A)}\Big]^* & = -A_{cf}A_{ff}^{-1} , \hspace{12ex} W_{\textnormal{ideal}}{(A)} = -A_{ff}^{-1}A_{fc}, \nonumber\\
\Big[Z_{\textnormal{ideal}}{(AA^*)}\Big]^* & = -(A_{cf}A_{ff}^* + A_{cc}A_{fc}^*)(A_{ff}A_{ff}^* + A_{fc}A_{fc}^*)^{-1}, \nonumber\\
W_{\textnormal{ideal}}{(A^*A)} & = -(A_{ff}^*A_{ff} + A_{cf}^*A_{cf})^{-1}(A_{ff}^*A_{fc} + A_{cf}^*A_{cc}). \nonumber\\
Z_{\textnormal{ideal}}{(A^{-1})} & = A_{cc}^{-1}A_{cf} , \hspace{12.5ex} W_{\textnormal{ideal}}{(A^{-1})} = A_{fc}A_{cc}^{-1}.\nonumber
\end{align}
Note that for general $Q$, $Z_{\textnormal{ideal}}{(Q^*)} =
\left[W_{\textnormal{ideal}}{(Q)}\right]^*$.

\Cref{fig:ideal} demonstrates the symmetry in ideal operators and orthogonal projections, by showing pairs of $R$ and $P$ such that $\Pi$ is orthogonal in the $I$-, $A$-, and $A^*A$-norms. For example, $R_{ideal}(AA^*)$ and $P_{ideal}(A)$ correspond to an orthogonal projection in the $I$-norm. All relations follow from \Cref{lem:ideal_pairs}.

\begin{figure}
\begin{center}
\begin{tikzpicture}
\node at (0,6) {$R_{\textnormal{ideal}}(\cdot)$};
\node at (4,6) {$P_{\textnormal{ideal}}(\cdot)$};
\draw [thick] (-1, 5.5) -- (5, 5.5);
\node at (0,5) {$A^{-*}$};
\node at (0,4) {$I$};
\node at (0,3) {$A$};
\node at (0,2) {$AA^*$};
\node at (4,5) {$A^{-*}$};
\node at (4,4) {$I$};
\node at (4,3) {$A$};
\node at (4,2) {$A^{*}A$};
\draw [dotted, ultra thick] (0.5, 5) -- (3.5,5);
\draw [dotted, ultra thick] (0.5, 4) -- (3.5,4);
\draw [dotted, ultra thick] (0.5, 3) -- (3.5,3);
\draw [dotted, ultra thick] (0.5, 2) -- (3.5,2);
\draw [dashed, ultra thick] (0.5, 5) -- (3.5,4);
\draw [dashed, ultra thick] (0.5, 4) -- (3.5,3);
\draw [dashed, ultra thick] (0.5, 3) -- (3.5,2);
\draw [ultra thick] (0.5, 4) -- (3.5,5);
\draw [ultra thick] (0.5, 3) -- (3.5,4);
\draw [ultra thick] (0.5, 2) -- (3.5,3);
\draw [ultra thick] (6, 4.25) -- (7, 4.25);
\node [align=right] at (9.1, 4.25) {$I$-orthogonal projection};
\draw [dotted, ultra thick] (6, 3.5) -- (7, 3.5);
\node [align=right] at (9.35, 3.5) {$A$-orthogonal projection};
\draw [dashed, ultra thick] (6, 2.75) -- (7, 2.75);
\node [align=right] at (9.2, 2.75) {$A^*A$-orthogonal projection};
\end{tikzpicture}
\end{center}
\caption{}
\label{fig:ideal}
\end{figure}
\end{remark}

\begin{remark}[Uniqueness]\label{rem:unique}
For a given SPD $M$, $R$ and $P$ form an $M$-orthogonal projection if they satisfy \eqref{eq:ortho1}. Let us fix $M$ and $P$. If there exists $R$ such that $\Pi$ is $M$-orthogonal, then $R$ is unique. Similarly, if we fix $M$ and $R$, if there exists $P$ such that $\Pi$ is $M$-orthogonal, this $P$ is unique. Although the pair of matrices $R$ and $P$ that project orthogonally onto $\mathcal{R}(P)$ in a given norm are unique, the norm itself is not unique, and a pair $R$ and $P$ can produce an $M$-orthogonal
$\Pi$ for many different $M$. For example, recall that $R= R_{ideal}(A)$ and interpolation $P= P_{ideal}(I)$, (that is, with $W = \mathbf{0}$) is orthogonal in the $I$-norm. It is easily verified that
this pair is also orthogonal in the norm induced by 
\begin{align*}
M & = \begin{bmatrix}M_{ff} & \mathbf{0} \\ \mathbf{0} & M_{cc} \end{bmatrix}.
\end{align*}
for any SPD $M_{ff}$ and $M_{cc}$. In this vein, one can use \eqref{eq:ortho1} to find norms in which specific $R$ and $P$ are $M$-orthogonal.
\end{remark}
\begin{remark}[Nonsymmetric generalization of $A$-norm]
Interestingly, there does not appear to be a nonsymmetric generalization of
the $A$-norm such that for a single $M$, all four relations in \Cref{fig:ideal} hold.
With significant algebra following the above framework, one can derive
necessary and sufficient conditions on fixed $M$ such that the pairs
$\{R_{ideal}(A^{-*}), P_{ideal}(A^{-*})\}$, $\{R_{ideal}(I), P_{ideal}(I)\}$,
and $\{R_{ideal}(A), P_{ideal}(A)\}$ each correspond to an $M$-orthogonal
coarse-grid correction, and these conditions do not have a general solution.
\end{remark}

\begin{remark}[Change of basis]\label{rem:basis}
As clear from \Cref{lem:ortho1}, results here are not constrained to the case of CF-splitting AMG with identities over the C-point block in transfer operators \eqref{eq:blocks}. Suppose $R$ and $P$ take the more general forms
\begin{equation*}
R = \begin{bmatrix} Z \\ Y\end{bmatrix}, \hspace{4ex}P=\begin{bmatrix} W \\ V \end{bmatrix}.
\end{equation*} 
We can still appeal to \Cref{lem:ortho1} and expand as we have done for various norms. However, as long as $Y$ and $V$ are non singular, the resulting $\Pi$ remains unchanged, and, thus, this more general form offers no advantages in theory. In practice, however, this general form introduces additional freedom compared with the classical AMG approach of $Y = V = I$. One potential benefit is if the $Z$ and $W$ leading to an orthogonal projection are not sparse in the classical CF-AMG setting, it is possible that general non-identity $Y$ and $V$ over C-point blocks will allow a more sparse or more computable/easier to approximate pair of transfer operators that yield an orthogonal projection. For example, consider \eqref{eq:ideal-Ainv}. The $A_{cc}^{-*}$ terms make these operators dense and generally intractable to form directly. But, if we apply $A_{cc}$ and $A_{cc}^*$ as the change of bases (i.e., non identity C-point blocks) to $R$ and $P$, respectively, it is trivial to construct the orthogonal pair exactly. 
\end{remark}
\begin{remark}[Stability and approximation properties]
From the operator pairs in \Cref{fig:ideal}, it is easy to see that an
approximation property, on $R$ or $P$, and stability (or an orthogonal
projection) does not imply an approximation property on the other transfer
operator. For many scalar problems, ideal interpolation and restriction have
good approximation properties, but in general $Z=\mathbf{0}$ or $W =
\mathbf{0}$ have no approximation property, despite coupling for an $I$-
or $A^*A$-orthogonal projection with $R_{\textnormal{ideal}}$ or
$P_{\textnormal{ideal}}$, respectively. Moreover, these pairings coupled
with convergent F-relaxation yield a convergent two-grid method, indicating
that approximation properties on both $P$ and $R$ are not necessary for
convergence. Conversely, in \cite{nonsymm}, an example is used
to prove that an approximation property on both $P$ \textit{and} $R$ is also
not sufficient to guarantee a stable coarse-grid correction. Thus,
the three measures, stability and two approximation properties, are
largely independent and each must be given proper consideration, and
approximation properties on $P$ and $R$ are neither necessary nor
sufficient for two-grid convergence or stable $\Pi$.
\end{remark}
\begin{remark}[Nonsymmetric AIR AMG]
As discussed in \Cref{sec:intro}, in nonsymmetric problems, relaxation must
quickly attenuate error that is amplified by the non-orthogonal coarse-grid
correction. It turns out that the nonsymmetric AMG method AIR \cite{air1,air2} is designed in such a way that this property
can be satisfied. Recall from
\cite{air2}, error propagation of coarse-grid correction can be written as
\begin{align}
(I-\Pi)\mathbf{e} & = \begin{pmatrix}\mathbf{e}_f - W\mathbf{e}_c \\\mathbf{0}\end{pmatrix} - P(R^*AP)^{-1}R^*A
	\begin{pmatrix}\mathbf{e}_f - W\mathbf{e}_c \\ \mathbf{0}\end{pmatrix}. \label{eq:cpt} 
\end{align}
If $R = R_{\textnormal{ideal}}$, then $R^*A \begin{bmatrix}\mathbf{x} \\
\mathbf{0}\end{bmatrix} = \mathbf{0}$ for all $\mathbf{x}$, and the latter term
in \eqref{eq:cpt} is zero. Error is then zero at C-points, and only nonzero at
F-points.  For $W = \mathbf{0}$, the projection is orthogonal in the $I$-norm, 
while for $W\neq\mathbf{0}$, error can  be amplified only at F-points. If  $R \neq
R_{\textnormal{ideal}}$, but is a very good approximation, the second term is
likely to remain small, and error (particularly that amplified by the
non-orthogonal projection) will still primarily reside on F-points. Indeed,
AIR is used with a post-F-relaxation (sometimes followed by some
combination of F/C-relaxation), which exactly focuses on the space in which
the non-orthogonal projection amplifies error (that is, on F-points). Thus,
for well-conditioned $A_{ff}$, F-relaxation coupled with a good approximation to
$R_{\textnormal{ideal}}$ is very likely to be complementary in exactly
the way we need it, and provides an intuitive understanding on the
effectiveness of AIR applied to hyperbolic and advection-dominated problems.
\end{remark}

\subsection*{Acknowledgements}
Los Alamos National Laboratory report LA-UR-23-26552.

\bibliographystyle{siamplain}
\bibliography{references.bib}

\begin{thebibliography}{10}

\bibitem{brannick2018optimal}
{\sc J.~Brannick, F.~Cao, K.~Kahl, R.~D. Falgout, and X.~Hu}, {\em Optimal
  interpolation and compatible relaxation in classical algebraic multigrid},
  SIAM Journal on Scientific Computing, 40 (2018), pp.~A1473--A1493.

\bibitem{Brannick.Falgout.2010}
{\sc J.~J. Brannick and R.~D. Falgout}, {\em {Compatible relaxation and
  coarsening in algebraic multigrid}}, SIAM Journal on Scientific Computing, 32
  (2010), pp.~1393 -- 1416,
  \href{http://dx.doi.org/10.1137/090772216}{doi:\nolinkurl{10.1137/090772216}}.

\bibitem{Brezina:2010dm}
{\sc M.~Brezina, T.~A. Manteuffel, S.~F. McCormick, J.~W. Ruge, and
  G.~Sanders}, {\em {Towards Adaptive Smoothed Aggregation ($\alpha$SA) for
  Nonsymmetric Problems}}, SIAM Journal on Scientific Computing, 32 (2010),
  pp.~14--39.

\bibitem{Deutsch:1995tc}
{\sc F.~Deutsch}, {\em {The angle between subspaces of a Hilbert space}}, NATO
  ASI Series C Mathematical and Physical {\ldots},  (1995), pp.~107--130.

\bibitem{Falgout:2005hm}
{\sc R.~D. Falgout, P.~S. Vassilevski, and L.~T. Zikatanov}, {\em {On two-grid
  convergence estimates}}, Numerical Linear Algebra with Applications, 12
  (2005), pp.~471--494.

\bibitem{Livne.Livne.2004}
{\sc O.~E. Livne}, {\em {Coarsening by compatible relaxation}}, Numerical
  Linear Algebra with Applications,  (2004).

\bibitem{Lottes:2017jz}
{\sc J.~Lottes}, {\em {Towards Robust Algebraic Multigrid Methods for
  Nonsymmetric Problems}}, Springer Theses, Springer International Publishing,
  Cham, 2017.

\bibitem{air1}
{\sc T.~A. Manteuffel, S.~M\"unzenmaier, J.~Ruge, and B.~S. Southworth}, {\em
  Nonsymmetric reduction-based algebraic multigrid}, SIAM J. Sci. Comput., 41
  (2019), pp.~S242--S268,
  \href{http://dx.doi.org/10.1137/18M1193761}{doi:\nolinkurl{10.1137/18M1193761}}.

\bibitem{Manteuffel:2017}
{\sc T.~A. Manteuffel, L.~N. Olson, J.~B. Schroder, and B.~S. Southworth}, {\em
  A root-node based algebraic multigrid method}, SIAM Journal on Scientific
  Computing, 39 (2017), pp.~S723--S756.

\bibitem{air2}
{\sc T.~A. Manteuffel, J.~W. Ruge, and B.~S. Southworth}, {\em Nonsymmetric
  algebraic multigrid based on local approximate ideal restriction
  ({$\ell${AIR}})}, SIAM Journal on Scientific Computing, 40 (2018),
  pp.~A4105--A4130.

\bibitem{nonsymm}
{\sc T.~A. Manteuffel and B.~S. Southworth}, {\em Convergence in norm of
  nonsymmetric algebraic multigrid}, SIAM Journal on Scientific Computing, 41
  (2019), pp.~S269--S296.

\bibitem{Notay:2000vy}
{\sc Y.~Notay}, {\em {A robust algebraic multilevel preconditioner for
  non-symmetric M-matrices}}, Numerical Linear Algebra with Applications, 7
  (2000), pp.~243--267.

\bibitem{Notay:2010em}
{\sc Y.~Notay}, {\em {Algebraic analysis of two-grid methods: The nonsymmetric
  case}}, Numerical Linear Algebra with Applications, 17 (2010), pp.~73--96.

\bibitem{notay2018}
{\sc Y.~Notay}, {\em Analysis of two-grid methods: The nonnormal case}, Tech.
  Report GANMN 18-01, 2018.

\bibitem{Sala:2008cv}
{\sc M.~Sala and R.~S. Tuminaro}, {\em {A New Petrov{\textendash}Galerkin
  Smoothed Aggregation Preconditioner for Nonsymmetric Linear Systems}}, SIAM
  Journal on Scientific Computing, 31 (2008), pp.~143--166.

\bibitem{Szyld:2006bg}
{\sc D.~B. Szyld}, {\em {The many proofs of an identity on the norm of oblique
  projections}}, Numerical Algorithms, 42 (2006), pp.~309--323.

\bibitem{vassilevski2008multilevel}
{\sc P.~S. Vassilevski}, {\em Multilevel block factorization preconditioners:
  Matrix-based analysis and algorithms for solving finite element equations},
  Springer Science \& Business Media, 2008.

\bibitem{Wiesner:2014cy}
{\sc T.~A. Wiesner, R.~S. Tuminaro, W.~A. Wall, and M.~W. Gee}, {\em {Multigrid
  transfers for nonsymmetric systems based on Schur complements and Galerkin
  projections}}, Numerical Linear Algebra with Applications, 21 (2013),
  pp.~415--438.

\end{thebibliography}

\end{document}